\documentclass[a4paper,12pt]{article}
\usepackage{times,url}
\usepackage{amsmath,amssymb,amsthm,amsfonts}

\newtheorem{theorem}{Theorem}[section]
\newtheorem{lemma}{Lemma}[section]

\newtheorem{remark}{Remark}[section]

\begin{document}
\setcounter{page}{1} 
\vspace{10mm}

\begin{center}
{\LARGE \bf  Conditions Equivalent to the \\
Descartes-Frenicle-Sorli Conjecture \\
on Odd Perfect Numbers}
\vspace{8mm}

{\large \bf Jose Arnaldo B. Dris}
\vspace{3mm}

Department of Mathematics and Physics, Far Eastern University \\ 
Nicanor Reyes Street, Sampaloc, Manila, Philippines \\
e-mail: \url{jadris@feu.edu.ph}
\vspace{2mm}

\end{center}
\vspace{10mm}

\noindent
{\bf Abstract:} The Descartes-Frenicle-Sorli conjecture predicts that $k=1$ if $q^k n^2$ is an odd perfect number with Euler prime $q$.  In this note, we present some conditions equivalent to this conjecture. 
\vspace{10mm}

\section{Introduction}
If $N$ is a positive integer, then we write $\sigma(N)$ for the sum of the divisors of $N$.  A number $N$ is \emph{perfect} if $\sigma(N)=2N$.  We denote the abundancy index $I$ of the positive integer $w$ as $I(w) = \displaystyle\frac{\sigma(w)}{w}$.  We also denote the deficiency $D$ of the positive integer $x$ as $D(x) = 2x - \sigma(x)$ \cite{OEIS-A033879}.

Euclid and Euler showed that that an even perfect number $E$ must have the form
$$E=\left(2^p - 1\right){2^{p-1}}$$
where $2^p - 1$ is a \emph{Mersenne prime}.  On the other hand, Euler showed that an odd perfect number $O$ must have the form
$$O=q^k n^2$$
where $q$ is an \emph{Euler prime} (i.e., $q \equiv k \equiv 1 \pmod 4$ and $\gcd(q,n)=1$).

It is currently unknown whether there are any odd perfect numbers.  On the other hand, only $49$ even perfect numbers have been found, a couple of which were discovered by the Great Internet Mersenne Prime Search \cite{GIMPS}.  It is conjectured that there are infinitely many even perfect numbers, and that there are no odd perfect numbers.

Descartes, Frenicle and subsequently Sorli conjectured that $k=1$ \cite{Beasley}.  Sorli conjectured $k=1$ after testing large numbers with eight distinct prime factors for perfection \cite{Sorli}.

Holdener presented some conditions equivalent to the existence of odd perfect numbers in \cite{Holdener}.  In this paper, we prove the following results:

\begin{lemma}
\label{lem:1}
If $N=q^k n^2$ is an odd perfect number with Euler prime $q$, then $k=1$ if and only if
$$\frac{\sigma(n^2)}{q} \mid n^2.$$
\end{lemma}

\begin{lemma}
\label{lem:2}
If $N=q^k n^2$ is an odd perfect number with Euler prime $q$, then
$$I(n^2) \leq 2 - \frac{5}{3q}.$$
\end{lemma}

\begin{lemma}
\label{lem:3}
If $N=q^k n^2$ is an odd perfect number with Euler prime $q$, then $k=1$ if and only if
$$D(n^2) \mid n^2.$$
\end{lemma}

\begin{theorem}
\label{thm:1}
If $N=q^k n^2$ is an odd perfect number with Euler prime $q$, then
$$I(n^2) = 2 - \frac{5}{3q}$$
holds if and only if $k=1$ and $q=5$.
\end{theorem}

All of the proofs given in this note are elementary.

\section{Preliminaries}
Let $N=q^k n^2$ be an odd perfect number with Euler prime $q$.

First, we show that the following equations hold.

\begin{lemma}
\label{lem:4}
If $N=q^k n^2$ is an odd perfect number with Euler prime $q$, then
$$\gcd\left(n^2, \sigma(n^2)\right) = \frac{D(n^2)}{\sigma(q^{k-1})} = \frac{\sigma(N/q^k)}{q^k}.$$
\end{lemma}

\begin{proof}
Since $N=q^k n^2$ is an odd perfect number, we have
$$\sigma(q^k)\sigma(n^2) = \sigma(N) = 2N = 2{q^k}{n^2},$$
from which it follows that $q^k \mid \sigma(n^2)$ (because $\gcd\left(q^k,\sigma(q^k)\right)=1$). Hence,
$$\frac{\sigma(n^2)}{q^k} = \frac{\sigma(N/q^k)}{q^k}$$
is an integer.

First, we prove that
$$\frac{D(n^2)}{\sigma(q^{k-1})} = \frac{\sigma(N/q^k)}{q^k}.$$
We rewrite the equation
$$\sigma(q^k)\sigma(n^2) = 2{q^k}{n^2}$$
as
$$\left(q^k + \sigma(q^{k-1})\right)\sigma(n^2) = 2{q^k}{n^2}$$
$$\sigma(q^{k-1})\sigma(n^2) = {q^k}\left(2n^2 - \sigma(n^2)\right) = {q^k}\cdot{D(n^2)}$$
$$\frac{\sigma(n^2)}{q^k} = \frac{D(n^2)}{\sigma(q^{k-1})},$$
and we are done.

Next, we show that
$$\gcd\left(n^2, \sigma(n^2)\right) = \frac{D(n^2)}{\sigma(q^{k-1})}.$$
We already know that
$$\sigma(n^2) = {q^k}\cdot{\left(\frac{D(n^2)}{\sigma(q^{k-1})}\right)}.$$
Since $\sigma(q^k)\sigma(n^2) = 2{q^k}{n^2}$, we also obtain
$$\frac{2n^2}{\sigma(q^k)} = \frac{\sigma(n^2)}{q^k} = \frac{D(n^2)}{\sigma(q^{k-1})}.$$
This implies that
$$n^2 = \frac{\sigma(q^k)}{2}\cdot{\left(\frac{D(n^2)}{\sigma(q^{k-1})}\right)}.$$
It follows that
$$\gcd\left(n^2, \sigma(n^2)\right) = \frac{D(n^2)}{\sigma(q^{k-1})}$$
since
$$\gcd\left(q^k, \frac{\sigma(q^k)}{2}\right) = \gcd(q^k, \sigma(q^k)) = 1.$$
This concludes the proof.
\end{proof}

\begin{remark}
\label{rem:1}
Dris obtained the lower bound $3$ for $\sigma(N/q^k)/q^k$ in \cite{Dris1} and \cite{Dris2}.  The following papers obtain (ever-increasing) lower bounds for this quantity: \cite{Dris-Luca}, \cite{Chen-Chen1}, \cite{Broughan-Delbourgo-Zhou}, \cite{Chen-Chen2}.
\end{remark}

\begin{remark}
\label{rem:2}
Notice that
$$\frac{\sigma(n^2)}{q^k} = \frac{2n^2}{\sigma(q^k)} > \frac{8}{5}\cdot\left(\frac{n^2}{q^k}\right)$$
since $I(q^k) < 5/4$ holds unconditionally (i.e., for $k \geq 1$).  Additionally, note that
$$\frac{8}{5}\cdot\left(\frac{n^2}{q^k}\right) > \frac{8n}{5}$$
is true if $q^k < n$.

Dris conjectured in \cite{Dris1} that $q^k < n$.  Recently, Brown has announced a proof for $q<n$, and that $q^k < n$ holds ``in many cases" \cite{Brown}.
\end{remark}

\begin{remark}
\label{rem:3}
It is an easy exercise to prove that $q^k < n$ implies the biconditional        
$$q^k < n \Leftrightarrow \sigma(q^k) < \sigma(n) \Leftrightarrow \frac{\sigma(q^k)}{n} < \frac{\sigma(n)}{q^k}.$$
We refer the interested reader to MSE (\url{http://math.stackexchange.com/q/713035}) for an expository proof.
\end{remark}

Next, we prove the following lemmas.

\begin{lemma}
\label{lem:5}
If $N=q^k n^2$ is an odd perfect number with Euler prime $q$, then
$$I(n^2) \geq 2-\frac{5}{3q} \Rightarrow \left(k=1 \land q=5\right).$$
\end{lemma}

\begin{proof}
Note that
$$I(n^2) = \frac{2}{I(q^k)} = \frac{2q^k(q-1)}{q^{k+1}-1} = 2-2\cdot\left(\frac{q^k-1}{q^{k+1}-1}\right).$$
If
$$I(n^2) \geq 2-\frac{5}{3q}$$
then we obtain
$$2-2\cdot\left(\frac{q^k-1}{q^{k+1}-1}\right) \geq 2-\frac{5}{3q}$$
$$\frac{5}{3q} \geq {2}\cdot\left(\frac{q^k-1}{q^{k+1}-1}\right)$$
$$5q^{k+1}-5 \geq 6q^{k+1}-6q$$
$$0 \geq q^{k+1}-6q+5,$$
which then implies that $k=1$.  (Otherwise, if $k>1$ we have
$$0 \geq q^{k+1}-6q+5 \geq {q^6}-6q+5$$
since $k \equiv 1 \pmod 4$, contradicting $q \geq 5$.)
Now, since $k=1$, we get
$$0 \geq {q^2}-6q+5=(q-5)(q-1)$$
which implies that $1 \leq q \leq 5$.  Together with $q \geq 5$, this means that $q=5$.  This concludes the proof.
\end{proof}

\begin{lemma}
\label{lem:6}
If $N=q^k n^2$ is an odd perfect number with Euler prime $q$, then $k=1$ implies
$$I(n^2) \leq 2-\frac{5}{3q}.$$
\end{lemma}

\begin{proof}
Suppose that $k=1$. By Lemma \ref{lem:1}, we have $\sigma(n^2)/q \mid n^2$.  This implies that there exists an (odd) integer $d$ such that
$$n^2 = {d}\cdot{\left(\frac{\sigma(n^2)}{q}\right)}.$$
Note that, from the equation $\sigma(N)=2N$, we obtain (upon setting $k=1$)
$$(q+1)\sigma(n^2)=\sigma(q)\sigma(n^2)=2qn^2$$
from which we get
$$d = \frac{n^2}{\sigma(n^2)/q} = \frac{q+1}{2}.$$
Notice that, when $k=1$, we can derive
$$\frac{5}{3} \leq I(n^2) = \frac{2}{I(q)} = \frac{2q}{q+1} < 2$$
so that we have
$$\frac{q}{2} < d = \frac{q}{I(n^2)} \leq \frac{3q}{5}.$$

Additionally, note that, when $k=1$, we have
$$I(n^2) = \frac{2}{I(q)} = \frac{2q}{q+1} = \frac{2q+2}{q+1} - \frac{2}{q+1} = 2 - \frac{1}{\frac{q+1}{2}} = 2 - \frac{1}{d}.$$

Consequently, we obtain
$$\frac{q}{2} < d \leq \frac{3q}{5}$$
$$\frac{5}{3q} \leq \frac{1}{d} < \frac{2}{q}$$
$$2 - \frac{2}{q} < 2 - \frac{1}{d} = I(n^2) \leq 2 - \frac{5}{3q},$$
and we are done.
\end{proof}

\section{The proof of Lemma \ref{lem:1}}
Let $N=q^k n^2$ be an odd perfect number with Euler prime $q$.

By Lemma \ref{lem:4}, we have
$$\frac{D(n^2)}{\sigma(q^{k-1})} = \frac{\sigma(N/q^k)}{q^k}.$$
This equation can be rewritten as
$$D(n^2) = \frac{\sigma(n^2)}{q}\cdot{I(q^{k-1})}.$$

Suppose that $\sigma(n^2)/q \mid n^2$.  Trivially, we know that $\sigma(n^2)/q \mid \sigma(n^2)$.  Thus, we have
$$\frac{\sigma(n^2)}{q} \mid \left(2n^2 - \sigma(n^2)\right) = D(n^2),$$
giving
$$\frac{\sigma(n^2)}{q} \mid \frac{\sigma(n^2)}{q}\cdot{I(q^{k-1})}.$$
This implies that $I(q^{k-1})$ is an integer.  Since $1 \leq I(q^{k-1}) < 5/4$, we obtain $k=1$.

Now assume that $k=1$.  We obtain
$$2n^2 - \sigma(n^2) = D(n^2) = \frac{\sigma(n^2)}{q}.$$
Again, since $\sigma(n^2)/q \mid \sigma(n^2)$, this implies
$$\frac{\sigma(n^2)}{q} \mid n^2$$
since $\sigma(n^2)/q$ is odd.

This concludes the proof of Lemma \ref{lem:1}.

\section{The proof of Lemma \ref{lem:2}}
Let $N=q^k n^2$ be an odd perfect number with Euler prime $q$.

Assume to the contrary that
$$I(n^2) > 2-\frac{5}{3q}.$$
Following the proof of Lemma \ref{lem:5}, we get
$$0 > q^{k+1} - 6q + 5.$$
Since $k \equiv 1 \pmod 4$, then $k \geq 1$, which implies that
$$0 > q^{k+1} - 6q + 5 \geq q^2 - 6q + 5 = (q - 5)(q - 1).$$
This then finally gives $1 < q < 5$, contradicting $q \geq 5$.

We therefore conclude that
$$I(n^2) \leq 2-\frac{5}{3q},$$
and this finishes the proof of Lemma \ref{lem:2}.

\section{The proof of Lemma \ref{lem:3}}
Let $N=q^k n^2$ be an odd perfect number with Euler prime $q$.

By Lemma \ref{lem:4}, we have
$$\frac{D(n^2)}{\sigma(q^{k-1})} = \frac{2n^2}{\sigma(q^k)}.$$
Multiplying throughout the last equation by $\sigma(q^{k-1})\sigma(q^k)$, we get
$$D(n^2)\sigma(q^k) = {2n^2}\sigma(q^{k-1}).$$
If $k=1$, then it is evident that $D(n^2) \mid 2n^2$, from which it follows that $D(n^2) \mid n^2$ since $D(n^2)$ is odd.

Now, assume that $D(n^2) \mid n^2$.  Then we have
$$\frac{\sigma(q^k)}{2\sigma(q^{k-1})} = \frac{n^2}{D(n^2)}$$
is an integer.  Since $\gcd\left(\sigma(q^{k-1}),\sigma(q^k)\right) = 1$, the previous equation then implies that $k=1$.

This concludes the proof of Lemma \ref{lem:3}.  In particular, we have shown that the Descartes-Frenicle-Sorli conjecture for odd perfect numbers $q^k n^2$ is true if and only if the non-Euler part $n^2$ is \emph{deficient-perfect} \cite{OEIS-A271816}.

\section{The proof of Theorem \ref{thm:1}}
Let $N=q^k n^2$ be an odd perfect number with Euler prime $q$.

We want to prove that the equation
$$I(n^2) = 2-\frac{5}{3q}$$
holds if and only if $k=1$ and $q=5$.

Suppose that
$$I(n^2) = 2-\frac{5}{3q}.$$

Following the proof of Lemma \ref{lem:4}, we get
$$0 = q^{k+1} - 6q + 5.$$

Assume to the contrary that $k>1$. Since $k \equiv 1 \pmod 4$, we obtain
$$0 = q^{k+1} - 6q + 5 \geq q^6 - 6q + 5.$$
This contradicts $q \geq 5$.  Thus, we have established that $k=1$.

Substituting $k=1$ into $0 = q^{k+1} - 6q + 5$, we have
$$0 = q^2 - 6q + 5 = (q - 5)(q - 1)$$
which implies that $q=5$ since $q \geq 5$.  This takes care of one direction of Theorem \ref{thm:1}.

For the other direction, assume that $k=1$ and $q=5$.  We want to show that
$$I(n^2) = 2-\frac{5}{3q}.$$
Note that, when $k=1$ and $q=5$, we obtain
$$I(n^2) = \frac{2}{I(q)} = \frac{2q}{q+1} = \frac{5}{3}.$$
We also get
$$2-\frac{5}{3q}=2-\frac{1}{3}=\frac{5}{3}$$
so that we have
$$I(n^2) = 2-\frac{5}{3q},$$
as desired.

\section{Concluding Remarks}
We end with some remarks related to the biconditional
$$k=1 \Longleftrightarrow \bigg(D(n^2) \mid n^2\bigg).$$

Suppose that $k=1$. By Lemma \ref{lem:3} and Lemma \ref{lem:4}, we obtain
$$D(n^2) = \gcd\left(n^2, \sigma(n^2)\right) = \frac{\sigma(n^2)}{q} = \frac{n^2}{(q+1)/2}.$$
Multiplying throughout the equations by $q(q+1)/2$, we have
$${D(n^2)}\cdot\bigg({\frac{q(q+1)}{2}}\bigg) = \bigg(\frac{q+1}{2}\bigg)\cdot{\sigma(n^2)} = qn^2 = N.$$

In fact, as shown by Slowak \cite{Slowak}, every odd perfect number $N$ has the form
$$N = {q^k}\cdot\frac{\sigma(q^k)}{2}\cdot{d}$$
for some $d > 1$.  We give a quick proof of this fact here.

By Lemma \ref{lem:4}, we obtain
$$\frac{D(n^2)}{\sigma(q^{k-1})} = \gcd\left(n^2, \sigma(n^2)\right) = \frac{\sigma(n^2)}{q^k} = \frac{n^2}{\sigma(q^k)/2}.$$
Multiplying throughout the equations by ${q^k}\sigma(q^k)/2$, we get
$$\frac{{q^k}\sigma(q^k)}{2}\cdot\frac{D(n^2)}{\sigma(q^{k-1})} = \frac{{q^k}\sigma(q^k)}{2}\cdot\gcd\left(n^2, \sigma(n^2)\right) = {q^k}{n^2} = N,$$
where
$$d = \frac{D(n^2)}{\sigma(q^{k-1})} = \gcd(n^2, \sigma(n^2)) > 1$$
by Remark \ref{rem:1}.

\section{Acknowledgments}
The author thanks Keneth Adrian P. Dagal for helpful conversations that led to most of the results presented in this paper.  The author is also indebted to the anonymous referee(s) whose valuable feedback improved the overall presentation and style of this manuscript.

\bibliographystyle{amsplain}

\end{document}